\documentclass[a4,12pt]{amsart}
\usepackage{amsfonts}
\usepackage{amssymb}
\usepackage{amsmath}
\usepackage{amsthm}
\usepackage{fullpage}
\theoremstyle{plain}
\newtheorem{thm}{Theorem}[section]
\newtheorem{lem}[thm]{Lemma}
\theoremstyle{definition}

\newtheorem{ex}[thm]{Example}
\newtheorem{rem}[thm]{Remark}

\newtheorem{pro}[thm]{Proposition}
\numberwithin{equation}{section}

\begin{document}
\title[Neumann systems with dependence on the gradient]{Multiple positive radial solutions for Neumann elliptic systems with gradient dependence}\thanks{Partially supported by G.N.A.M.P.A. - INdAM (Italy)}
\author[F. Cianciaruso]{Filomena Cianciaruso}
\address{Filomena Cianciaruso, Dipartimento di Matematica ed Informatica, Universit\`{a}
della Calabria, 87036 Arcavacata di Rende, Cosenza, Italy}
\email{cianciaruso@unical.it}
\author[G. Infante]{Gennaro Infante}%
\address{Gennaro Infante, Dipartimento di Matematica e Informatica, Universit\`{a} della
Calabria, 87036 Arcavacata di Rende, Cosenza, Italy}
\email{gennaro.infante@unical.it}
\author[P. Pietramala]{Paolamaria Pietramala}%
\address{Paolamaria Pietramala, Dipartimento di Matematica e Informatica, Universit\`{a}
della Calabria, 87036 Arcavacata di Rende, Cosenza, Italy}
\email{pietramala@unical.it}
\begin{abstract}
We provide new results on the existence, non-existence
and multiplicity of non-negative radial solutions for semilinear
elliptic systems with Neumann boundary conditions on an annulus.
Our approach is topological and relies on the classical fixed
point index. We present an example to illustrate our theory.
\end{abstract}
\maketitle
\section{Introduction}
In the recent manuscript~\cite{defig-ubi}, by means of topological methods, De Figueiredo and Ubilla studied the existence of positive radially symmetric solutions of the following system of elliptic differential equations, with dependence on the gradient, under~\emph{Dirichlet} boundary conditions (BCs)
\begin{equation*}
\begin{cases}
-\Delta u =f_1(|x|,u,v,|\nabla u|,|\nabla v|)\ \text{in}\ \Omega, \\
-\Delta v=f_2(|x|,u,v,|\nabla u|,|\nabla v|)\ \text{in}\ \Omega,\\
\ u=v=0
 \text{ on }\partial\Omega,
\end{cases}
\end{equation*}
where $\Omega$ is an annular domain.

On the other hand, there has been recent interest in the existence of radial solutions of elliptic equations with~\emph{Neumann} BCs on
annular domains, see for example~\cite{bonheureannulus, boscaggin1, Feltrin-Zan, sfecci, Sovr-Zan} and references therein.
 In particular, variational methods are used in~\cite{bonheureannulus}, shooting methods are utilized in ~\cite{boscaggin1, Sovr-Zan}, upper and lower solution techniques are employed in~\cite{sfecci} and 
 Mawhin's coincidence degree theory is used in~\cite{Feltrin-Zan}. We mention that the existence of multiple solutions has been investigated,  on more general domains, by means of critical point theory in~\cite{Bonanno1, Bonanno2, dagui}.
We note that the nonlinearities considered in~\cite{Bonanno1, Bonanno2, bonheureannulus, boscaggin1, dagui, Feltrin-Zan, sfecci, Sovr-Zan} do not depend on the gradient.

Here we focus on the systems of BVPs
\begin{equation}\label{ellbvp-secapp}
\begin{cases}\
-\Delta u =f_1(|x|,u,v,|\nabla u|,|\nabla v|)\ \text{in}\ \Omega, \\
-\Delta v=f_2(|x|,u,v,|\nabla u|,|\nabla v|)\ \text{in}\ \Omega,\\
\ \displaystyle\frac{\partial u}{\partial r}=\frac{\partial
v}{\partial r}=0
 \text{ on }\partial\Omega\,,
\end{cases}
\end{equation}
where $\Omega=\{ x\in\mathbb{R}^n : R_0<|x|<R_1\}$ is an annulus,
$0<R_0<R_1<+\infty$, the functions $f_i$ are continuous and $\dfrac{\partial}{\partial r}$ denotes (as
in ~\cite{nirenberg}) differentiation in the radial direction
$r=|x|$.

One difficulty that occurs when dealing with the existence of radial solutions of the system~\eqref{ellbvp-secapp}, is that the linear part of the associated ordinary differential equations is not invertible. 

In order to overcome this problem, we make use of a shift argument utilized, for example, in~\cite{Tor} (in the case of periodic BCs) and in~\cite{Han, infpieto, WebbZima}. In particular, here we benefit of the results given in~\cite{infpieto}, where the authors studied the existence of multiple solutions of one ordinary differential equation under local and nonlocal Neumann BCs. We stress that the results in~\cite{infpieto} do not allow a derivative dependence in the nonlinearity. In order to deal with the derivative dependence we construct a cone in in the space $C^1$.
The idea here is to involve a lower
bound for the function $u$ in terms of the $C^1$-norm, rather than
employing two Harnack-type inequalities on $u$ and $u'$ as
in~\cite{fm+rs}. As far as we are aware of, the cone we use is new. 

We prove the existence of multiple non-negative radial
solutions for the system~\eqref{ellbvp-secapp} by means of
classical fixed point index theory. 
We also prove, by an elementary argument, a non-existence result. We provide an example to illustrate the theory, where we show the existence of three non-negative, non-constant solutions.

\section{Preliminary results}
We begin by considering  in $\mathbb{R}^n$, $n\ge 2$, the equation
\begin{equation}\label{eqell}
-\triangle w=f(|x|,w,|\nabla w|)\text{ in } \Omega.
\end{equation}
Since we are looking for the existence of non-negative radial solutions $w=w(r)$, $r=|x|$ of the system \eqref{ellbvp-secapp},
we rewrite \eqref{eqell} for $w$ in polar coordinates as
\begin{equation}\label{eqinterm}
-w''(r)-\dfrac{n-1}{r}w'(r)= f(r,w(r),|w'(r)|) \quad\text{ in }
[R_{0}, R_{1}].
\end{equation}
Set $w(t)=w(r(t))$, where, for $t\in[0,1]$, (see \cite{defig-ubi, dolo3})
\begin{equation*}
r(t):=\begin{cases}
R_1^{1-t}R_0^{t},\ &n=2,\\
\left(\frac{A}{B-t}\right)^{\frac{1}{n-2}},\ &n\geq 3,
\end{cases}
\end{equation*}
$$
A=\frac{(R_0R_1)^{n-2}}{R_1^{n-2}-R_0^{n-2}}\,\,\,\,\text{
 and  }\,\,\,\,B=\frac{R_1^{n-2}}{R_1^{n-2}-R_0^{n-2}}\,.
$$
Take, for $t\in[0,1]$,
\begin{equation*}
d(t):=\begin{cases}
r^2(t) \log^2(R_1/R_0), \ & n=2,\\
\left(\frac{R_0R_1\left(R_1^{n-2}-R_0^{n-2}\right)}{n-2}\right)^2\,\frac{1}{\left(R_1^{n-2}-(R_1^{n-2}-R_0^{n-2})t\right)^{\frac{2(n-1)}{n-2}}},\
&n\geq 3,
\end{cases}
\end{equation*}
then the equation ~\eqref{eqinterm} becomes
\begin{equation*}
-w''(t)= d(t)
f\left(r(t),w(t),\left|\frac{w'(t)}{r'(t)}\right|\right).
\end{equation*}

Note that, since $\lambda=0$ is an eigenvalue of the associated linear problem
$$
-w''(t)=\lambda w(t),\quad w'(0)=w'(1)=0,
$$
the corresponding Green's function does not exist. Therefore we
proceed as  in~\cite{infpieto} and we study a related BVP for which the Green's
function can be constructed; namely, fixed $\omega>0$, for $t \in
[0,1]$, we set
\begin{gather}\label{shiftint2}
- w''(t)+ \omega ^2 w(t)=g(t,w(t),|w'(t)|):=d(t)  f\left(r(t),w(t),\left|\frac{w'(t)}{r'(t)}\right|\right)+\omega^2 w(t),\\
\nonumber
 w'(0)=w'(1)=0.
\end{gather}
\begin{rem}
The choice of an appropriate $\omega$ in the shift argument above is somewhat delicate, as it affects some important constants that we use in our theory (see~\eqref{mM}-\eqref{ci}). Also, given a nonlinearity $f$, $\omega$ should be large enough to allow the auxiliary function $g$ occurring in~\eqref{shiftint2} to be non-negative.   
\end{rem}

Moving our attention back to the system~\eqref{ellbvp-secapp}, in a similar way as above, by setting $u(t)=u(r(t))$ and $v(t)=v(r(t))$, we can associate
to the system~\eqref{ellbvp-secapp}
the system of ODEs
\begin{equation}\label{1syst}
\begin{cases}
-u''(t) + \omega_1^2u(t)= g_1(t,u(t),v(t),|u'(t)|,|v'(t)|) \  \text{ in } [0,1], \\
-v''(t) +\omega_2^2v(t)=g_2(t,u(t),v(t),|u'(t)|,|v'(t)|) \  \text{ in } [0,1],\\
\ u'(0)=u'(1)=v'(0)=v'(1)=0.
\end{cases}
\end{equation}
 By a radial solution of the system~\eqref{ellbvp-secapp} we mean a solution of the system~\eqref{1syst}.

From now on we assume that:
\begin{itemize}
\item[$(H)$] For $i=1,2$, $f_i:[R_0,R_1]\times([0,+\infty[)^4\rightarrow \mathbb{R}$  is a continuous function such that
\begin{equation}\label{lwb}
f_i(r,z_1,z_2,w_1,w_2)\geq \,-\frac{\omega_i^2\,z_i}{\displaystyle\max_{t \in
[0,1]}d(t)}\ \text{ in } [R_0,R_1]\times([0,+\infty[)^4.
\end{equation}
\end{itemize}
\begin{rem}
Due to the assumption $(H)$, $g_i $ are non-negative continuous functions in~$[0,1]\times([0,+\infty[)^4.$ 
Note that if we fix  $\Omega=\{ x\in\mathbb{R}^2 : 1<|x|<e\}$ and $\omega_i=1$, then~\eqref{lwb} reads
\begin{equation*}
f_i(r,z_1,z_2,w_1,w_2)\geq \,-e^{-2}z_i\ \text{ in } [1,e]\times([0,+\infty[)^4,
\end{equation*}
a linear bound from below for the nonlinearities.
\end{rem}

We  study the existence of solutions of the system~\eqref{1syst}
by means of the fixed points of a suitable operator on the space
$C^1[0,1]\times C^1[0,1]$ equipped with the norm
\begin{equation*}
|| (u,v)|| :=\max \{|| u|| _{C^{1}},|| v|| _{C^{1}}\},
\end{equation*}%
where $|| w|| _{C^{1}}:=\max \left\{ || w|| _{\infty},|| w'|| _{\infty}\right\} $ and $|| y|| _{\infty}:=\underset{t\in [ 0,\,1]\,%
}{\max }|y(t)|$.

We define the integral operator $T:C^1[0,1]\times C^1[0,1]\to
C^1[0,1]\times C^1[0,1]$ by
\begin{equation}\label{operT}
\begin{array}{c}
T(u,v)(t):=\left(
\begin{array}{c}
T_{1}(u,v)(t) \\
T_{2}(u,v)(t)%
\end{array}%
\right)  =\left(
\begin{array}{c}
\int_{0}^{1}k_1(t,s)g_{1}(s,u(s),v(s),|u'(s)|,|v'(s)|)\,ds \\
\int_{0}^{1}k_2(t,s)g_2(s,u(s),v(s),|u'(s)|,|v'(s)|)ds%
\end{array}%
\right),%
\end{array}
\end{equation}%
where the Green's functions $k_i$  are given by
$$
k_i(t,s)=\frac{1}{\omega_i\sinh\omega_i}\,\begin{cases}
\cosh(\omega_i(1-t))\cosh \omega_i s, &0\leq s\leq t\leq 1,\cr
\cosh(\omega_i(1-s))\cosh \omega_i t, & 0\leq t\leq s\leq
1.
\end{cases}
$$

The following Lemma provides some useful properties of the kernels $k_i$
(see~\cite{infpieto} for some of these estimates).
\begin{lem} The following hold, for $i=1,2$:{}
\begin{itemize}
 \item[(1)]  The kernel $k_i$ is positive and continuous in  $[0,1]\times  [0,1]$. Moreover we have
\begin{align*}
c_{k_i}\, \phi_i (s)\leq & k_i(t,s) \leq \phi_i(s),\ \text{ for } (t,s)\in
[0,1] \times [0,1],
\end{align*}
where we take
$$\phi_i(s):=\sup_{t\in [0,1]} k_i(t,s)= k_i(s,s)=\frac{1}{\omega_i\sinh\omega_i}\,\cosh(\omega_i(1-s))\cosh \omega_i s\,
$$
and
$$
c_{k_i}:=\min_{t\in[0,1]}\min_{s\in[0,1]}\frac{k_i(t,s)}{\phi_i(s)}=\frac{1}{\cosh
\omega_i}\,.
$$
 \item[(2)]  The function $k_i(\cdot,s)$ is differentiable in $[0,1]$ for a.e. $s\in [0,1]$, with
$$
\dfrac{\partial k_i}{\partial t}(t,s)=\frac{1}{\sinh
\omega_i}\,\begin{cases} -\cosh \omega_i s \sinh(\omega_i(1-t)),&0\leq
s< t\leq 1,\cr \cosh(\omega_i(1-s))\sinh\omega_i t, & 0\leq t< s\leq
1,\cr\end{cases}
$$
and for every $\tau \in [ 0,1]$ we have
\begin{equation*}
\lim_{t\rightarrow \tau }\left| \frac{\partial k_i}{\partial
t}(t,s)- \frac{\partial k_i}{\partial t}(\tau ,s)\right| =0,\
\text{for a.e.}\ s\in [ 0,1].
\end{equation*}
The partial derivative  $\dfrac{\partial k_i}{\partial t}(t,s)<0$
for $s<t$, $\dfrac{\partial k_i}{\partial t}(t,s)>0$ for $s>t$ and
$$
 \left|\frac{\partial k_i}{\partial t}(t,s)\right|\leq \omega_i\phi_i(s),\text{    for   }t\in [0,1]\text{     and a.e. }\,s\in [0,1].
 $$
\end{itemize}
\end{lem}
By direct calculation  we obtain
\begin{equation}\label{mM} 
m_i:=\left(\sup_{t\in [0,1]}\int_0^1k_i(t,s)\,ds\right)^{-1}=\omega_i^2,
\
M_i:=\left(\inf_{t\in[0,1]}\int_0^1k_i(t,s)ds\right)^{-1}=\omega_i^2
\end{equation} 
 and
\begin{equation}\label{mstar}
m_i^*:= \left(\sup_ {t\in [ 0,1]}\int_{0}^{1}\left | \frac{\partial
k_i}{\partial t}(t,s)\right |\,ds\right)^{-1} = \frac{\omega_i \sinh
(\omega_i)}{2 \, \sinh ^2(\omega_i/2)}.
\end{equation}

For $i=1,2$, set
\begin{equation}\label{ci}
c_i:=c_{k_i}\cdot \min\{1, \omega_i^{-1}\}
\end{equation}
and 
consider the cone in $ C^{1}[0,1]$
\begin{equation*}
K_{i}:=\Bigl\{ w\in C^{1}[0,\,1]: \,\,\min _{t\in [ 0,1]} w(t)\geq
c_i \| w\|_{C^{1}}\Bigr\},
\end{equation*}
which is 
similar to a cone of non-negative functions used by
Krasnosel'ski\u\i{} and Guo in the space $C[0,1]$, see for
example~\cite{guolak, krzab}. 
We define 
 the cone in $C^{1}[0,1]\times C^{1}[0,1]$
\begin{equation*}
K:=\{(u,v)\in K_{1}\times K_{2}\}.%
\end{equation*}

We have the following result.
\begin{pro}
The operator $T$ leaves the cone $K$ invariant and is compact.
\end{pro}
\begin{proof}
 Let $r>0$
and take $(u,v)\in K$ such that $\|(u,v)\| \leq r$. Then we have,
for $t \in [0,1]$,
\begin{align*}
0\leq T_i(u,v)(t)
&\leq \int_{0}^{1}\phi_i(s)g_i(s,u(s),v(s),|u'(s)|,|v'(s)|)ds
\end{align*}
and
\begin{align*}
\left|\left(
T_i(u,v)\right)'(t)\right|&=\left|\int_{0}^{1}\frac{\partial
k_i}{\partial t}(t,s)g_i(s,u(s),v(s),|u'(s)|,|v'(s)|)ds\right|\\&\le
\int_{0}^{1}\left|
\frac{\partial k_i}{\partial t}(t,s)\right|g_i(s,u(s),v(s),|u'(s)|,|v'(s)|)ds\\
&\leq \omega_i\int_{0}^{1}\phi_i(s)g_{i}(s,u(s),v(s),|u'(s)|,|v'(s)|)ds.
\end{align*}
Consequently we have
\begin{align*}
|| T_i(u,v)|| _{C^1}&\leq \max\{1,\omega_i\}\int_{0}^{1}\phi_i(s)g_i(s,u(s),v(s),|u'(s)|,|v'(s)|)ds.
\end{align*}
Moreover, for $t \in [0,1]$, we get
\begin{equation*}
T_i(u,v)(t)\geq
c_{k_i}\int_{0}^{1}\phi_i(s)g_i(s,u(s),v(s),|u'(s)|,|v'(s)|)\,ds;
\end{equation*}
and thus we obtain
\begin{equation}\label{min}
\min_{t\in [0,1]}T_i(u,v)(t) \geq
c_{k_i}\int_{0}^{1}\phi_i(s)g_i(s,u(s),v(s),|u'(s)|,|v'(s)|)\,ds \geq c_i
||T_i(u,v)|| _{C^1}\,.
\end{equation}

Since \eqref{min} holds for every $r>0$, we have
$T_{i}K_{i}\subset K_{i}$.
By the properties of the Green's functions $k_i$ and using the
Arzel\`{a}-Ascoli Theorem, we obtain the compactness of the operator~$T$.
\end{proof}

We now recall some results regarding the classical fixed point index (more details can be found, for example, in~\cite{Amann-rev, guolak}).
\begin{thm}\label{index}
Let $K$ be a cone in an ordered real Banach space $X$. Let $\Omega $ be
an open bounded subset of $X$ with $0 \in \Omega\cap K$ and
$\overline{\Omega \cap K}\neq K$.
Assume that $F:\overline{\Omega \cap K}\to K$ is a compact map such that $x\neq Fx$ for all $x\in \partial (\Omega \cap K)$. Then the fixed point index $i_{K}(F, \Omega \cap K)$ has the following properties.
\begin{itemize}
\item[(1)] If there exists $h\in K\setminus \{0\}$ such that $x\neq Fx+\lambda h$ for all $x\in \partial (\Omega \cap K)$ and all $\lambda>0$,
then $i_{K}(F, \Omega \cap K)=0$.
\item[(2)] If  $\mu x \neq Fx$ for all $x\in \partial (\Omega \cap K)$ and for every $\mu \geq 1$, then $i_{K}(F, \Omega \cap K)=1$.
\item[(3)] If $i_K(F,\Omega \cap K)\ne0$, then $F$ has a fixed point in $\Omega \cap K$.
\item[(4)] Let $\Omega^{1}$ be open in $X$ with $\overline{\Omega^{1}\cap K}\subset \Omega \cap K$. If $i_{K}(F, \Omega \cap K)=1$ and $i_{K}(F, \Omega^{1}\cap K)=0$, then $F$ has a fixed point in $(\Omega \cap K)\setminus \overline{\Omega^{1}\cap K}$. The same result holds if $i_{K}(F, \Omega \cap K)=0$ and $i_{K}(F, \Omega^{1}\cap K)=1$.
\end{itemize}
\end{thm}

For our index calculations we make use of the following open bounded sets (relative to~$K$), namely, for
$\rho_1,\rho_2>0$,
\begin{equation*}
K_{\rho _{1},\rho _{2}}:=\{(u,v)\in K:|| u|| _{C^{1}}<\rho _{1}\
\text{ and }\ || v|| _{C^{1}}<\rho _{2}\},
\end{equation*}
\begin{equation*}
 V_{\rho _{1},\rho _{2}}:=\{(u,v)\in K:
\displaystyle{\min_{t\in [0,1]}}u(t)<\rho_1 \text{ and
}\displaystyle{\min_{t\in [0,1]}}v(t)<\rho_2\}.
\end{equation*}
\begin{lem}
The sets $K_{\rho _{1},\rho _{2}}$ and $ V_{\rho _{1},\rho _{2}}$ have the following properties:
\begin{itemize}
\item[$(P_1)$]$K_{\rho_1,\rho_2}\subset V_{\rho_1,\rho_2}\subset K_{\rho_1/c_1,\rho_2/c_2}$.
\item[$(P_2)$] $(w_1,w_2)\in\partial K_{\rho _{1},\rho _{2}}$ if and only if
$(w_{1},w_{2})\in K$ and for some $i\in \{1,2\}$
\begin{itemize}
\item[] $\|w_i\|_{\infty}=\rho_i$,\text{} $c_i \rho_i \le
w_i(t)\le \rho_i$ for $t\in[0,1]$  and $-\rho_j\leq
w'_j(t)\leq \rho_j$ for $j=1,2$ and $t\in[0,1]$,
\end{itemize}
or
\begin{itemize}
\item[]
$\|w'_i\|_{\infty}=\rho _{i}$, $0\leq w_j(t)\leq \rho_j$
and $-\rho_j\leq w'_j(t)\leq \rho_j$ for
$j=1,2$ and $t\in[0,1]$.
\end{itemize}
\item[$(P_3)$] $(w_1,w_2) \in \partial V_{\rho_1,\rho_2}$ if and only if  $(w_1,w_2)\in K$ and for some $i\in \{1,2\}$ $\displaystyle
\min_{t\in [0,1]} w_i(t)= \rho_i$  and
 $\rho_i \le w_i(t) \le \rho_i/c_i$  for  $t\in[0,1]$.
\end{itemize}
\end{lem}

\section{Existence and non-existence of solutions}
In this Section we establish some existence and non-existence
results for the semilinear elliptic system
\begin{equation}\label{PDE}
\begin{cases}\
-\Delta u =f_1(|x|,u,v,|\nabla u|,|\nabla v|)\ \text{in}\ \Omega, \\
-\Delta v=f_2(|x|,u,v,|\nabla u|,|\nabla v|)\ \text{in}\ \Omega,\\
\ \displaystyle\frac{\partial u}{\partial r}=\frac{\partial
v}{\partial r}=0
\ \text{on}\ \partial\Omega\,,
\end{cases}
\end{equation}
where $\Omega=\{ x\in\mathbb{R}^n : R_0<|x|< R_1\}$,
$0<R_0<R_1<+\infty$.

We define the following
sets:
\begin{align*}
\tilde{\Omega}_1^{\rho_1,\rho_2}&=[ R_0,R_1]\times \left [\rho_1,
\frac{\rho_1}{c_1}\right]\times\left [0,
\frac{\rho_2}{c_2}\right]\times \left[0,
\frac{\alpha\rho_1}{c_1}\right]\times\left [ 0,
\frac{\alpha\rho_2}{c_2}\right],\\
\tilde{\Omega}_2^{\rho_1,\rho_2}&=[ R_0,R_1]\times \left [0,
\frac{\rho_1}{c_1}\right]\times\left [\rho_2,
\frac{\rho_2}{c_2}\right]\times \left[0,
\frac{\alpha\rho_1}{c_1}\right]\times\left [ 0,
\frac{\alpha\rho_2}{c_2}\right],\\
\tilde{A}_1^{s_1,s_2}&=[
R_0,R_1]\times[c_1s_1,s_1]\times[0,s_2]\times\left[0,
\alpha s_1\right]\times\left[0,\alpha s_2\right],\\
\tilde{A}_2^{s_1,s_2}&=[R_0,R_1]\times[0,s_1]\times[c_2s_2,s_2]\times\left[0,
\alpha s_1\right]\times\left[0, \alpha s_2\right],
\end{align*}
where $\alpha:=\displaystyle\inf_{t \in [0,1]}|r'(t)|$ and $c_i$ is given by~\eqref{ci}.

The first existence result is the following.

\begin{thm}\label{ellyptic}
 Suppose that condition $(H)$ is satisfied and there exist
$\rho_1, \rho_2, s_1,s_2\in (0,+\infty )$, with $\rho _{i}/c_i<c_is
_{i}\,,\,\,i=1,2$, such that the following conditions hold
\begin{equation}\label{pde1}
\inf_{\tilde{\Omega}_i^{\rho_1,\rho_2}} f_i(r, w_1,w_2,z_1,z_2)>0,
\end{equation}
\begin{equation}\label{pde2}
\sup_{\tilde{A}_i^{s_1,s_2}}
f_i(r,w_1,w_2,z_1,z_2)<\frac{(\min\{m_i,m_i^*\}-\omega_i^2)}{\displaystyle\sup_{t
\in [0,1]}d(t)}\,s_i.
\end{equation}

Then the system~\eqref{PDE} has at least one non-negative radial solution.
\end{thm}

\begin{proof}
Firstly, we note that the choice of the numbers $\rho_i$ and $s_i$
assures the compatibility of the conditions \eqref{pde1} and
\eqref{pde2}. Moreover, since
$$g_i(t,u(t),v(t),|u'(t)|,|v'(t)|)=d(t)f_i\left(r(t),u(t),v(t),\left|\frac{u'(t)}{r'(t)}\right|,\left|\frac{v'(t)}{r'(t)}\right|\right)+\omega_i^2u(t),$$
when $f_i$ acts on
$\tilde{\Omega}_i^{\rho_1,\rho_2}$ and on $\tilde{A}_i^{s_1,s_2}$,
$g_i$ acts respectively on the sets $\Omega_i^{\rho_1,\rho_2}$
and $A_i^{s_1,s_2}$, defined by
\begin{align*}
\Omega_1^{\rho_1,\rho_2}&=[ 0,1]\times \left [\rho_1,
\frac{\rho_1}{c_1}\right]\times\left [0,
\frac{\rho_2}{c_2}\right]\times \left[0,
\frac{\rho_1}{c_1}\right]\times\left [ 0,
\frac{\rho_2}{c_2}\right],\\
\Omega_2^{\rho_1,\rho_2}&=[ 0,1]\times \left [0,
\frac{\rho_1}{c_1}\right]\times\left [\rho_2,
\frac{\rho_2}{c_2}\right]\times \left[0, \frac{\rho_1}{c_1}\right]
\times\left [ 0,
\frac{\rho_2}{c_2}\right],\\
A_1^{s_1,s_2}&=[0,1]\times[c_1s_1,s_1]\times[0,s_2]\times[0,s_1]\times[0,s_2],\\
A_2^{s_1,s_2}&=[0,1]\times[0,s_1]\times[c_2s_2,s_2]\times[0,
s_1]\times[0, s_2].
\end{align*}

By using Theorem~\ref{index}, we show that the compact operator
$T$ defined in \eqref{operT} has a fixed point in
$K_{s_1,s_2}\setminus \overline{V}_{\rho_1,\rho_2}$.

Consider $h(t)=1$ for $t\in [ 0,1]$ and note that $(h,h)\in K$.
Firstly we claim that
\begin{equation*}
(u,v)\neq T(u,v)+\lambda (h,h),\  \text{for}\ (u,v)\in \partial
V_{\rho_1,\rho_2}\ \text{and}\ \lambda \geq 0,
\end{equation*}%
that assures that $i_K(T, V_{\rho_1,\rho_2})=0$.

Assume, by contradiction, that there exist $(u,v)\in
\partial V_{\rho_1,\rho_2}$ and $\lambda \geq 0$ such that
$(u,v)=T(u,v)+\lambda (h,h)$.
Without loss of generality, by $(P_3)$ we can assume that, for $t\in [0,1]$,
$$\min_{t\in[0,1]} u(t)=
\rho_1,\,\rho_1\leq u(t)\leq {\rho_1/c_1}\text{ and }
-{\rho_1/c_1}\le u'(t)\leq {\rho_1/c_1},
$$
$$
0\leq v(t)\leq {\rho_2/c_2}\text{ and }-{\rho_2/c_2} \leq v'(t)\leq
{\rho_2/c_2}.
$$
Therefore we consider, for $t\in [ 0,1]$,
\begin{equation}\label{eq}
u(t) =
\int_{0}^{1}k_1(t,s)g_1(s,u(s),v(s),|u'(s)|,|v'(s)|)ds+{\lambda}.
\end{equation}

By (\ref{pde1}) in $\Omega_1^{\rho_1,\rho_2}$ we have
$$
g_1(t,u(t),v(t),|u'(t)|,|v'(t)|)>\omega_1^2u(t)
$$
and, consequently, we have
\begin{equation}\label{index0}
\min_{\Omega_1^{\rho_1,\rho_2}} g_1(t,
u(t),v(t),|u'(t)|,|v'(t)|)>\omega_1^2\min_{t \in
[0,1]}u(t)=\omega_1^2\rho_1\,.
\end{equation}
Taking the minimum over $[0,1]$ in \eqref{eq} and using
(\ref{index0}), we obtain
\begin{equation*}
\rho_1=\min_{t\in [ 0,1]}u(t)>\omega_1^2 \rho_1\min_{t \in
[0,1]}\int_0^1 k_1(t,s)ds+{\lambda }=\omega_1^2 \rho_1
\frac{1}{\omega_1^2}+\lambda=\rho_1+\lambda,
\end{equation*}%
i.e. a contradiction. The case of $\min_{t\in[0,1]} v(t)=\rho_2$ follows in a similar manner.

We claim now that $\lambda (u,v)\neq T(u,v)$ for every $(u,v)\in
\partial K_{s_1,s_2}$ and for every $\lambda \geq 1$,
which implies that $i_{K}(T,K_{s_1,s_2})=1$.

Assume this is not
true. Then there exist $\lambda \geq 1$ and $(u,v)\in
\partial K_{s_1,s_2}$ such that $\lambda (u,v)=T(u,v)$.
Suppose that
\begin{equation*}
|| u || _{\infty}=s_1,\,\, \,|| u' ||_{\infty}\leq s_1,\,\,\,||
v|| _{\infty}\leq s_2 \text{  and   }|| v'|| _{\infty}\leq s_2.
\end{equation*}

Consider, for $t\in [0,1]$,
\begin{equation}\label{u}
\lambda u(t)= \int_{0}^{1} k_1(t,s) g_1(s,u(s),v(s),|u'(s)|,
|v'(s)|)ds.
\end{equation}

Note that, by (\ref{pde2}), we have, in $A_1^{s_1,s_2}$,
\begin{align*}
g_1(t,u(t),v(t),|u'(t)|,|v'(t)|)&<\frac{d(t)}{\displaystyle\sup_{t \in
[0,1]}d(t)}(\min\{m_1,m_1^*\}-\omega_1^2) s_1+\omega_1^2u(t)\\
&\leq (\min\{m_1,m_1^*\}-\omega_1^2)\, s_1+\omega_1^2u(t),
\end{align*}
and, consequently, we obtain
\begin{multline}\label{index1}
\max_{A_1^{s_1,s_2}}
g_1(t,u(t),v(t),|u'(t)|,|v'(t)|)\\
<(\min\{m_1,m_1^*\}-\omega_1^2)\,
s_1+\omega_1^2\max_{t \in [0,1]}u(t)=\min\{m_1,m_1^*\}s_1.
\end{multline}
Taking the maximum in $[0,1]$ in \eqref{u} and using
(\ref{index1}), we have
\begin{equation*}
\lambda\, s_1< m_1s_1\,\max_{t \in
[0,1]}\int_0^1k_1(t,s)ds=m_1s_1\frac{1}{m_1}=s_1,
\end{equation*}
i.e. $\lambda s_1<s_1,$ which contradicts the fact that $\lambda
\geq 1$.

If we have
\begin{equation*}
|| u || _{\infty}\leq s_1,\,\,\,|| u' ||_{\infty}=s_1,\,\,\,|| v
|| _{\infty}\leq s_2\,\,\, \text{   and   }\,\,\,|| v' ||
_{\infty}\leq s_2,
\end{equation*}%
then we obtain
\begin{align}\label{u'}
\lambda | u'(t) | &\leq \int_0^1\left | \frac{\partial k_1}{\partial
t}(t,s)\right | g_1(s,u(s),v(s),|u'(s)|,|v'(s)|)ds.
\end{align}
Taking the maximum in $[0,1]$ in \eqref{u'} and using the condition
 $(\ref{index1})$, we obtain
\begin{equation*}
\lambda\,s_1 < m_1^*s_1\, \max_{t\in [ 0,1]}\int_{0}^{1}\left |
\frac{\partial k_1}{\partial t}(t,s)\right |\,ds=m_1^*s_1\frac{1}{m_1^*}
=s_1,
\end{equation*}
a contradiction. The other two cases follow in a similar manner.

Therefore we have $i_K(T, V_{\rho_1,\rho_2})=0$ and $i_K(T,
K_{s_1,s_2})=1$.
Since $\displaystyle\frac{\rho_i}{c_i}<c_is_i<s_i$, $i=1,2$, by
property $(P_1)$ we have $\overline{V}_{\rho_1,\rho_2}\subset
\overline{K}_{\frac{\rho_1}{c_1},\frac{\rho_2}{c_2}}\subset
K_{s_1,s_2}$. It follows from Theorem \ref{index} that $T$ has a
fixed point $(\bar{u}, \bar{v})$ in $K_{s_1,s_2}\setminus
\overline{V}_{\rho_1,\rho_2}$.

Therefore the system~\eqref{PDE} admits a non-negative radial solution.
\end{proof}
\begin{rem}\label{rem2}
The conditions~\eqref{pde1} and~\eqref{pde2} are nothing but growth estimates on two boxes. 
If we fix  $\Omega=\{ x\in\mathbb{R}^2 : 1<|x|<e\}$ and $\omega_i=1$, then~\eqref{pde2} reads
\begin{equation*}
\sup_{\tilde{A}_i^{s_1,s_2}}
f_i(r,w_1,w_2,z_1,z_2)<0,
\end{equation*}
a sign condition on $f_i$ on the set $\tilde{A}_i^{s_1,s_2}$.
\end{rem}

We can obtain another existence result in which the constrains on the growth of one of the components is relaxed at the cost of having to deal with a larger domain.

We denote by
$\tilde{\Omega}^{*^{\rho_1,\rho_2}}$ the subset of $\mathbb
R^5$
\begin{align*}
\tilde{\Omega}^{*^{\rho_1,\rho_2}}&=[ R_0,R_1]\times \left [0,
\frac{\rho_1}{c_1}\right]\times\left [0,
\frac{\rho_2}{c_2}\right]\times \left[0,
\frac{\alpha\rho_1}{c_1}\right]\times\left [ 0,
\frac{\alpha\rho_2}{c_2}\right].
\end{align*}
\begin{thm}\label{ellyptic2} Suppose that condition $(H)$ is satisfied and there exist
$\rho_1, \rho_2, s_1,s_2\in (0,+\infty )$, with $\rho _{i}/c_i<c_is
_{i}\,,\,\,i=1,2$, such that the following conditions hold:
\begin{equation}\label{pde3}
\inf_{\tilde{\Omega}^{*^{\rho_1,\rho_2}}} f_i(r,
w_1,w_2,z_1,z_2)>\frac{\omega_i^2}{\displaystyle\inf_{t \in
[0,1]}d(t)}\,\rho_i,\ \mbox{for some}\ i=1,2,
\end{equation}
\begin{equation}\label{pde4}
\sup_{\tilde{A}_i^{s_1,s_2}}
f_i(r,w_1,w_2,z_1,z_2)<\frac{(\min\{m_i,m_i^*\}-\omega_i^2)}{\displaystyle\sup_{t
\in [0,1]}d(t)}\,s_i,  \mbox{ for }i=1,2.
\end{equation}

Then the system~\eqref{PDE} has at least one non-negative radial
solution.
\end{thm}
\begin{proof}
Note that, when $f_i$ acts on
$\tilde{\Omega}^{*^{\rho_1,\rho_2}}$, $g_i$ acts on the subset
\begin{align*}
\Omega^{*^{\rho_1,\rho_2}}&=[0,1]\times \left [0,
\frac{\rho_1}{c_1}\right]\times\left [0,
\frac{\rho_2}{c_2}\right]\times \left[0, \frac{\rho_1}{c_1}\right]
\times\left [0, \frac{\rho_2}{c_2}\right].
\end{align*}

Suppose that (\ref{pde3}) holds for $i=1$. Let $(u,v)\in \partial
V_{\rho_1,\rho_2}$ and $\lambda \geq 0$ such that
$(u,v)=T(u,v)+\lambda (1,1)$. Thus we have, for $t\in [0,1]$,
$$0\leq u(t)\leq {\rho_1/c_1},\,\,-{\rho_1/c_1} \le u'(t)\leq {\rho_1/c_1},\,\,0\leq v(t)\leq {\rho_2/c_2},\,\,-{\rho_2/c_2} \leq v'(t)\leq
{\rho_2/c_2}\,.
$$
By (\ref{pde3}) in $\Omega^{*^{\rho_1,\rho_2}}$ we have
\begin{align*}
g_1(t,u(t),v(t),|u'(t)|,|v'(t)|)>\frac{d(t)}{\displaystyle\inf_{t \in
[0,1]}d(t)}\omega_1^2\rho_1+\omega_1^2u(t)\geq\omega_1^2\rho_1+\omega_1^2u(t)
\end{align*}
and, consequently, we obtain
\begin{equation*}
\min_{\Omega^{*^{\rho_1,\rho_2}}} g_1(t,
u(t),v(t),|u'(t)|,|v'(t)|)>\omega_1^2\rho_1\,.
\end{equation*}

From now on,  one proceeds as in the proof of Theorem \ref{ellyptic}.
\end{proof}

By means of Theorem~\ref{index} and the results contained in Theorems~\ref{ellyptic} and~\ref{ellyptic2}, it is possible to state results regarding the existence of  \emph{several} non-negative solutions
of the system~\eqref{PDE}. For brevity, here we state a result regarding the existence of three non-negative solutions and refer to~\cite{lan, lan-lin-na,  lanwebb} for the other kind of results that can be stated.

\begin{thm}\label{multi2}
Suppose that condition $(H)$ is satisfied and there exist $\rho _{i},s_{i},\theta_i,\sigma_i\in
(0,+\infty )$ with $\rho _{i}<s _{i}$,  $s _{i}/c_i<c_i\theta
_{i}$ and $\theta_i<\sigma _{i}$, $i=1,2$,  such that
\begin{equation}\label{uno}
\sup_{\tilde{A}_i^{\rho_1,\rho_2}}
f_i(r,w_1,w_2,z_1,z_2)<\frac{(\min\{m_i,m_i^*\}-\omega_i^2)}{\displaystyle\sup_{t
\in [0,1]}d(t)}\,\rho_i, \mbox{ for }i=1,2,
\end{equation}
\begin{equation}\label{due}
\inf_{\tilde{\Omega}_i^{s_1,s_2}} f_i(r,
w_1,w_2,z_1,z_2)>0, \mbox{ for  }i=1,2,
\end{equation}
\begin{equation}\label{tre}
\sup_{\tilde{A}_i^{\theta_1,\theta_2}}
f_i(r,w_1,w_2,z_1,z_2)<\frac{(\min\{m_i,m_i^*\}-\omega_i^2)}{\displaystyle\sup_{t
\in [0,1]}d(t)}\,\theta_i, \mbox{ for }i=1,2
\end{equation}
and
\begin{equation}\label{quattro}
\inf_{\tilde{\Omega}_i^{\sigma_1,\sigma_2}} f_i(r,
w_1,w_2,z_1,z_2)>0, \mbox{ for }i=1,2,
\end{equation}
hold.

Then the system~\eqref{PDE} has at least three non-negative radial
solutions.
\end{thm}
\begin{proof}
As in the proof of Theorem~\ref{ellyptic}, the condition~\eqref{uno} implies that  $i_{K}(T,K_{\rho_1,\rho_2})=1$ and the condition~\eqref{due} provides that
$i_K(T, V_{s_1,s_2})=0$. This fact gives the existence of a first fixed point in  $V_{s_1,s_2}\setminus \overline{K}_{\rho_1,\rho_2}$. A second fixed point in  $K_{\theta_1,\theta_2}\setminus \overline{V}_{s_1,s_2}$ is given by the conditions~\eqref{tre} and ~\eqref{due}. Finally, the third fixed point is achieved in $V_{\sigma_1,\sigma_2}\setminus \overline{K}_{\theta_1,\theta_2}$ , by means of the conditions~\eqref{quattro} and~\eqref{tre}.
\end{proof}

In the following  example we show the applicability of Theorem~\eqref{multi2} and prove the existence of at least three non-negative, non-constant solutions.

\begin{ex}
Consider  in $\mathbb{R}^2$  the system of BVPs
\begin{gather}\label{ellbvpex}
\begin{cases}
-\Delta u =e^{-(|\nabla u|^2+|\nabla v|^2+6)}u(u-1-h(|x|))(u-2-h(|x|))(u-4-h(|x|))(2-\cos v)\text{ in } \Omega, \\
-\Delta v=e^{-(|\nabla u|^2+|\nabla v|^2+7)}v(v-1-h(|x|))(v-4-h(|x|))(v-7-h(|x|))(2-\sin u)\text{ in } \Omega,\\
\ \displaystyle\frac{\partial u}{\partial r}=\frac{\partial
v}{\partial r}=0
 \text{ on }\partial\Omega\,,
\end{cases}
\end{gather}
where $\Omega=\{ x\in\mathbb{R}^2 : 1<|x|<e\}$ and $h(|x|)=\displaystyle\frac{|x|^2}{333}$.

For these nonlinearities we take $\omega_{1}=\omega_{2}=1$ and by direct computation we get
\begin{align*}
 m_1=m_2=1, \,m_1^*=m_2^*=\frac{ \sinh
(1)}{2 \, \sinh ^2(1/2)}\cong 2.16,\, c_1=c_2=(\cosh 1)^{-1}\cong 0.64.
\end{align*}
With the choice of
$$
\rho_1=\rho_2=1/2,\,\, s_1=\frac{11}{10},\,\, s_2=2,\,\,
\theta_1=\frac{7}{2},\,\,\theta_2=\frac{13}{2},\,\,
\sigma_1=5,\,\, \sigma_2=8,
$$
we obtain, for $i=1,2,$
\begin{gather*}
\sup_{\tilde{A}_i^{\rho_1,\rho_2}}
f_i(r,w_1,w_2,z_1,z_2)<0, \quad 
\inf_{\tilde{\Omega}_i^{s_1,s_2}} f_i(r,
w_1,w_2,z_1,z_2)>0,\\
\sup_{\tilde{A}_i^{\theta_1,\theta_2}}
f_i(r,w_1,w_2,z_1,z_2)<0, \quad 
\inf_{\tilde{\Omega}_i^{\sigma_1,\sigma_2}} f_i(r,
w_1,w_2,z_1,z_2)>0.
\end{gather*}
Keeping in mind Remark~\ref{rem2}, the conditions \eqref{uno}-\eqref{quattro} of Theorem~\ref{multi2} are
satisfied; therefore the system~\eqref{ellbvpex} has at least
three nonnegative solutions. Note that, due to the presence of the term $h$ within the equations, the solutions have at least one non-constant component.
\end{ex}

We conclude by showing some non-existence results for the system~\eqref{PDE}.

\begin{thm}
Assume that condition $(H)$ is satisfied and one of following conditions holds:
\begin{equation}\label{cond1}
f_i(r,w_1,w_2,z_1,z_2)<0\,,\,\, r\in [R_0,R_1] \text{ and }
w_i>0,\,\,i=1,2,
\end{equation}
\begin{equation}\label{cond2}
f_i(r,w_1,w_2, z_1,z_2)>0\,\,,\,\, r\in [R_0,R_1]  \text{ and }
w_i>0\,,\,\,i=1,2.
\end{equation}

Then the only possible non-negative solution of the system \eqref{PDE} in $C^1[R_0,R_1]\times
C^1[R_0,R_1]$ is the trivial one.
\end{thm}

\begin{proof}
Suppose that (\ref{cond1}) holds and assume that there exists a
solution $(\bar{u},\bar{v})\in C^1[R_0,R_1]\times
C^1[R_0,R_1]$ of \eqref{PDE} $(\bar{u},\bar{v})\neq
(0,0)$; then, since $T$ maps the product of two cones of non-negative functions in $C^1[0,1]$ in $K$,
 $(u,v):=(\bar{u}\circ r,\bar{v}\circ r)$ is a fixed
point of $T$.

Let, for example, be $\|(u,v)\|=\|u\|_{C^1} \neq 0$
and, consequently, $\|u\|_{\infty}\neq 0$.
 Then, for $t\in [0,1]$, we have
\begin{multline*}
u(t)  =\int_0^1k_1(t,s)g_1(s,u(s),v(s),|u'(s)|,|v'(s)|)ds\\
=\int_0^1k_1(t,s)\left(d(t)f_1\left(r(s),u(s),v(s),\left|\frac{u'(s)}{r(s)}\right|,\left|\frac{v'(s)}{r(s)}\right|\right)+\omega_1^2u(s)\right)ds\\
<\omega_1^2\int_0^1k_1(t,s)u(s)ds \le
\omega_1^2\|u\|_\infty\int_0^1k_1(t,s)\,ds.
\end{multline*}
 Taking the maximum for $t\in [0,1]$, we have
$$
\|u\|_\infty<  \omega_1^2\|u\|_\infty\max_{t\in[0,1]}\int_0^1k_1(t,s)\,ds
=\|u\|_\infty,
$$
a contradiction.

Suppose that~\eqref{cond2} holds and assume that there exists
$(u,v)\in K$ such that $(u,v)=T(u,v)$ and $(u,v)\neq (0,0)$. Take,
for example, $\|u\|_{\infty}\neq 0$; then
$\sigma:=\displaystyle\min_{t\in[0,1]}u(t)>0$ since $u \in K_1$.  We
have, for $t\in [0,1]$,
\begin{align*}
&u(t)=  \int_0^1 k_1(t,s)\left(d(t)f_1\left(r(s),u(s),v(s),\left|\frac{u'(s)}{r(s)}\right|,\left|\frac{v'(s)}{r(s)}\right|\right)+\omega_1^2u(s)\right)ds\\
&>\omega_1^2 \int_0^1k_1(t,s)u(s)ds.
\end{align*}
Taking the minimum  for $t\in [0,1]$, we obtain
$$
\sigma=\min_{t\in[0,1]}u(t)> \omega_1^2\min_{t\in[0,1]}\int_{0}^{1}
k_1(t,s)u(s)\,ds \geq \omega_1^2\sigma \min_{t\in[0,1]}\int_{0}^{1} k_1(t,s)\,ds
=\sigma,
$$
a contradiction.
\end{proof}


\begin{thebibliography}{00}

\bibitem{Amann-rev} H. Amann,
Fixed point equations and nonlinear eigenvalue problems in ordered
Banach spaces, \textit{SIAM. Rev.}, \textbf{18} (1976), 620--709.

\bibitem{Bonanno1}
G. Bonanno and P. Candito, Three solutions to a Neumann problem for elliptic equations involving the $p$-Laplacian, \textit{Arch. Math. (Basel)}, \textbf{80} (2003), 424--429.

\bibitem{Bonanno2}
G. Bonanno and G. D'Agu\`\i , On the Neumann problem for elliptic equations involving the $p$-Laplacian, \textit{J. Math. Anal. Appl.}, \textbf{358} (2009),  223--228.

\bibitem{bonheureannulus}
D. Bonheure and  E. Serra, Multiple positive radial solutions on
annuli for nonlinear Neumann problems with large growth,
\textit{NoDEA Nonlinear Differential Equations Appl.},  \textbf{18}  (2011), 217--235.

\bibitem{boscaggin1}
A. Boscaggin, A note on a superlinear indefinite Neumann problem with multiple positive solutions, \textit{J. Math. Anal. Appl.}, \textbf{377} (2011), 259--268.

\bibitem{dagui}
G. D'Agu\`\i {} and G. Molica Bisci, Three non-zero solutions for elliptic Neumann problems, \textit{Anal. Appl. (Singap.)}, \textbf{9} (2011), 383--394.

\bibitem{defig-ubi} D. G. De Figueiredo and P. Ubilla,
Superlinear systems of second-order ODE's,
  \textit{Nonlinear Anal.}, \textbf{68}  (2008), 1765--1773.

\bibitem{Feltrin-Zan}
G. Feltrin and F. Zanolin, 
Existence of positive solutions in the superlinear case via coincidence degree: the Neumann and the periodic boundary value problems, 
\textit{Adv. Differential Equations}, \textbf{20} (2015), 937--982.

\bibitem{nirenberg} B. Gidas, W.-M. Ni and L. Nirenberg,
 Symmetry and related properties via the maximum principle,
\textit{Comm. Math. Phys.}, {\bf 68} (1979), 209--243.

\bibitem{guolak} D. Guo and V. Lakshmikantham,
\textit{Nonlinear Problems in Abstract Cones},
Academic Press, Boston, 1988.

\bibitem{Han} X. Han, Positive solutions for a three-point boundary value problems at resonance, \textit{J. Math. Anal. Appl.}, {\bf 336} (2007), 556--568.

\bibitem{infpieto}
G. Infante, P. Pietramala and F. A. F. Tojo, Non-trivial solutions
of local and non-local Neumann boundary-value problems,
\textit{Proc. Roy. Soc. Edinburgh Sect. A}, \textbf{146}
(2016), 337--369.

\bibitem{krzab} M. A. Krasnosel'ski\u{\i}{} and P. P. Zabre\u{\i}ko,
\textit{Geometrical methods of nonlinear analysis}, Springer-Verlag, Berlin, 1984.

\bibitem{lan} K. Q. Lan,
{Multiple positive solutions of semilinear differential equations
with singularities,} \textit{J. London Math. Soc.}, \textbf{63}
(2001), 690--704.

\bibitem{lan-lin-na} K. Q. Lan and W. Lin,
Positive solutions of systems of singular Hammerstein integral
equations with applications to semilinear elliptic equations in
annuli, \textit{Nonlinear Anal.}, \textbf{74} (2011), 7184--7197.

\bibitem{lanwebb} K. Q. Lan and J. R. L. Webb,
Positive solutions of semilinear differential equations with
singularities, \textit{J. Differential Equations}, \textbf{148}
(1998), 407--421.

\bibitem{fm+rs} F. Minh\'{o}s and R. de Sousa, On the solvability of
third-order three point systems of differential equations with dependence on
the first derivative, \textit{Bull. Braz. Math. Soc. (N.S.)},
\textbf{48} (2017),  485--503.

\bibitem{dolo3} J. M. do {\'O},  S. Lorca, J. S{\'a}nchez, and P. Ubilla,
Positive solutions for a class of multiparameter ordinary elliptic
systems, \textit{J. Math. Anal. Appl.}, \textbf{332} (2007),
1249--1266.

\bibitem{sfecci}
A. Sfecci, Nonresonance conditions for radial solutions of
nonlinear Neumann elliptic problems on annuli,
\textit{Rend. Istit. Mat. Univ. Trieste}, \textbf{46} (2014), 255--270.

\bibitem{Sovr-Zan}
E. Sovrano and F. Zanolin, Indefinite weight nonlinear problems with Neumann boundary conditions, \textit{J. Math. Anal. Appl.}, \textbf{452} (2017), 126--147.

\bibitem{Tor} P. J. Torres, Existence of one-signed periodic solutions of some second-order differential equations via a Krasnoselskii fixed point theorem, \textit{J. Differential Equations}, \textbf{190} (2003), 643--662.

\bibitem{WebbZima} J. R. L. Webb and M. Zima, Multiple positive solutions of resonant and non-resonant nonlocal boundary value problems, \textit{Nonlinear Anal.}, {\bf 71} (2009), 1369--1378.

\end{thebibliography}
\end{document}